\documentclass[letterpaper]{amsart}
 
\usepackage[margin=1in]{geometry} 
\usepackage{amsmath,amsthm,amssymb}
\usepackage[english]{babel}
\usepackage{caption}
\usepackage{subcaption}
\usepackage{tikz}
\usepackage{bm}
\usepackage{comment}
\usepackage{enumitem}
\setlist[itemize]{nolistsep}
\usepackage{hyperref}

\newtheorem{thm}{Theorem}
\newtheorem{lem}[thm]{Lemma}
\newtheorem{cor}[thm]{Corollary}
\newtheorem{conj}[thm]{Conjecture}
\newtheorem{prop}[thm]{Proposition}

\theoremstyle{definition}

\newtheorem{exmp}[thm]{Example}
\newtheorem{qn}{Question}

\theoremstyle{remark}

\DeclareMathOperator{\ssl}{\mathfrak{sl}}
\DeclareMathOperator{\Rest}{Res}
\DeclareMathOperator{\Rea}{Res^{\mathfrak{sl}_{r+1}}_{\mathfrak{sl}_{r}}}
\DeclareMathOperator{\Res}{Res^{G_2}_{A_2}}
\DeclareMathOperator{\Ress}{Res^{B_2}_{D_2}}
\DeclareMathOperator{\Resss}{Res^{B_3}_{D_3}}
\DeclareMathOperator{\Ree}{Res^{F_4}_{D_4}}
\DeclareMathOperator{\Rees}{Res^{F_4}_{B_4}}
\DeclareMathOperator{\Reese}{Res^{B_4}_{D_4}}
\DeclareMathOperator{\nores}{Res^{C_3}_{(A_1)^{\oplus 3}}}
\DeclareMathOperator{\mg}{\mathfrak{g}_2}
\DeclareMathOperator{\su}{\mathfrak{su}}
\DeclareMathOperator{\SU}{SU}

\DeclareMathOperator{\UU}{U}
\DeclareMathOperator{\OO}{O}
\DeclareMathOperator{\SP}{Sp}

\DeclareMathOperator{\so}{\mathfrak{so}}

\DeclareMathOperator{\g}{\mathfrak{g}}
\DeclareMathOperator{\lt}{\mathfrak{t}}
\DeclareMathOperator{\lp}{\mathfrak{p}}
\DeclareMathOperator{\Lie}{Lie}

\begin{document}
 
\title{Branching rules for splint root systems}
\author{Logan Crew, Alexandre A. Kirillov, and Yao-Rui Yeo}
\date{\today}

\address{Logan Crew, Department of Mathematics, University of Pennsylvania.}
\email{crewl@math.upenn.edu}

\address{Alexandre A. Kirillov, Department of Mathematics, University of Pennsylvania.}
\email{kirillov@math.upenn.edu}

\address{Yao-Rui Yeo, Department of Mathematics, University of Pennsylvania.}
\email{yeya@math.upenn.edu}


\begin{abstract}
A root system is \emph{splint} if it is a decomposition into a union of two root systems. Examples of such root systems arise naturally in studying embeddings of reductive Lie subalgebras into simple Lie algebras. Given a splint root system, one can try to understand its branching rule. In this paper we discuss methods to understand such branching rules, and give precise formulas for specific cases, including the restriction functor from the exceptional Lie algebra $\mathfrak{g}_2$ to $\mathfrak{sl}_3$.
\end{abstract}

\maketitle

\section{Background}
Branching rules in group representation theory are the mathematical counterpart of the phenomenon of ``broken symmetry" in physics. Gelfand-Tsetlin patterns \cite{GT1} yield a very transparent algorithm to describe the spectrum of the restriction of an irreducible representation of the ``big" group $G(n)$, which is either the unitary group $\UU(n)$ or the orthogonal group $\OO(n)$, to the ``small" group $G(n-1)$.\footnote{
In \cite{GT1}, Gelfand and Tsetlin published their formulas without proof, possibly because the paper was intended as a contribution to mathematical physics, and their proof may have been of a computational nature.}

The second author has formulated and popularized numerous concrete problems and approaches related to Gelfand-Tsetlin patterns. This resulted in the discovery of an analog of these patterns for symplectic groups $\SP(n)$ \cite{AAK,VS} (but not for exceptional groups) and also provided the foundation for the present collaboration. 

In the following, we give some context and motivation for our approach. Experimental data shows that for some $H \subseteq G$, the multiplicity coefficients $m_{\Lambda,\lambda}$ in the restriction formula
$$
\Rest^G_H \Pi_\Lambda=\sum_{\lambda\in\widehat{H}} m_{\Lambda,\lambda} \pi_\lambda\quad\text{for}\quad \Lambda\in\widehat{G}
$$
coincide with the weight multiplicities of some irreducible representation of an auxiliary group $K$ in a natural way.
Gelfand-Tsetlin patterns are a special case of this phenomenon; here $K$ is the direct product of several copies of $\SU(2)$. 

This could be expanded as follows: Since the Weyl character formula for a representation $\Pi$ of $G$ describes  the restriction of $\Pi$ to the maximal torus $T\subset G$, 
the observation above is reminiscent of the chain rule for the derivative of the composite map $F=f\circ g$, where we have
\begin{equation*}
\mathbf DF(x)=\mathbf Df(g(x))\mathbf Dg(x).
\end{equation*}
In our case the role of the composite function is played by the restriction functor which satisfies
$$
\Rest^G_T\,=\, \Rest^H_T\circ\Rest^G_ H.
$$
Moreover, the restriction functor is compatible with natural operations on representations (such as sums, tensor products, and symmetric and exterior powers). This suggests a possible direction for future research: to show that any functor with these properties and some ``boundary conditions'' must satisfy an analog of the chain rule in the form proposed in this paper.

There are several other ways to prove the formula: from a change of variables in the Weyl formula to using the integral formula for the character and geometry of co-adjoint orbits. 

\section{A case study} \label{sec:four}

Consider the following two tables of integers. Figure \ref{f1a} shows the table of dimensions of irreducible representations of $\mathfrak{sl}_3$ indexed by highest weight $(\alpha,\beta)$, and Figure \ref{f1b} is the corresponding table for the exceptional Lie algebra $\mathfrak{g}_2$ indexed by highest weight $(k,l)$. Let $A_{\alpha,\beta}$ be the integer at the $(\alpha,\beta)$-entry of the left table, and let $G_{k,l}$ be the integer at the $(k,l)$-entry of the right table. Then the explicit formulas for $A_{\alpha,\beta}$ and $G_{k,l}$ are as follows:
	\begin{align*}
    A_{\alpha,\beta} &= \frac{(\alpha+1)(\beta+1)(\alpha+\beta+2)}{2}, \\
    G_{k,l} &= \frac{(k+1)(k+l+2)(2k+3l+5)(k+2l+3)(k+3l+4)(l+1)}{120}.
    \end{align*}

    \begin{figure}[!htb]
\centering
\begin{subfigure}{.48\textwidth}
  \centering
    \begin{tikzpicture}[scale=0.85]
\draw[thick] (-0.5,-0.5) edge[->] (6.5,-0.5);
\draw[thick] (-0.5,-0.5) edge[->] (-0.5,6.5);
\draw (6, -1) node {$\alpha$};
\draw (-1, 6) node {$\beta$};
\draw (0,0) node {$1$};
\draw (1,0) node {$3$};
\draw (2,0) node {$6$};
\draw (3,0) node {$10$};
\draw (4,0) node {$15$};
\draw (5,0) node {$21$};
\draw (6,0) node {$28$};
\draw (0,1) node {$3$};
\draw (1,1) node {$8$};
\draw (2,1) node {$15$};
\draw (3,1) node {$24$};
\draw (4,1) node {$35$};
\draw (5,1) node {$48$};
\draw (6,1) node {$63$};
\draw (0,2) node {$6$};
\draw (1,2) node {$15$};
\draw (2,2) node {$27$};
\draw (3,2) node {$42$};
\draw (4,2) node {$60$};
\draw (5,2) node {$81$};
\draw (6,2) node {$105$};
\draw (0,3) node {$10$};
\draw (1,3) node {$24$};
\draw (2,3) node {$42$};
\draw (3,3) node {$64$};
\draw (4,3) node {$90$};
\draw (5,3) node {$120$};
\draw (6,3) node {$154$};
\draw (0,4) node {$15$};
\draw (1,4) node {$35$};
\draw (2,4) node {$60$};
\draw (3,4) node {$90$};
\draw (4,4) node {$125$};
\draw (5,4) node {$165$};
\draw (6,4) node {$210$};
\draw (0,5) node {$21$};
\draw (1,5) node {$48$};
\draw (2,5) node {$81$};
\draw (3,5) node {$120$};
\draw (4,5) node {$165$};
\draw (5,5) node {$216$};
\draw (6,5) node {$273$};
\draw (0,6) node {$28$};
\draw (1,6) node {$63$};
\draw (2,6) node {$105$};
\draw (3,6) node {$154$};
\draw (4,6) node {$210$};
\draw (5,6) node {$273$};
\draw (6,6) node {$343$};
\end{tikzpicture}
\caption{Dimensions $A_{\alpha, \beta}$ of irreducible representations of $\mathfrak{sl}_3$}
\label{f1a}
\end{subfigure}\hfill\begin{subfigure}{.48\textwidth}\centering
    \begin{tikzpicture}[scale=0.85]
\draw[thick] (-0.5,-0.5) edge[->] (3.5,-0.5);
\draw[thick] (-0.5,-0.5) edge[->] (-0.5,3.5);
\draw (3, -1) node {$k$};
\draw (-1, 3) node {$l$};
\draw (0,0) node {$1$};
\draw (1,0) node {$7$};
\draw (2,0) node {$27$};
\draw (3,0) node {$77$};
\draw (0,1) node {$14$};
\draw (1,1) node {$64$};
\draw (2,1) node {$189$};
\draw (3,1) node {$448$};
\draw (0,2) node {$77$};
\draw (1,2) node {$286$};
\draw (2,2) node {$729$};
\draw (3,2) node {$1547$};
\draw (0,3) node {$273$};
\draw (1,3) node {$896$};
\draw (2,3) node {$2079$};
\draw (3,3) node {$4096$};
\end{tikzpicture}
\caption{Dimensions $G_{k, l}$ of irreducible representations of $\mathfrak{g}_2$}\label{f1b}
    \end{subfigure}
    \caption{$A_{\alpha, \beta}$ and $G_{k, l}$ for small values}
    \end{figure}
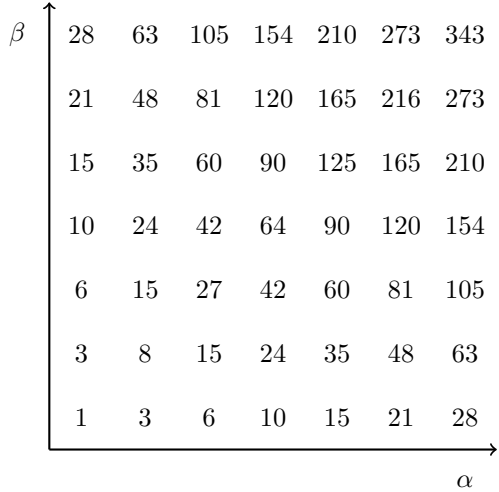
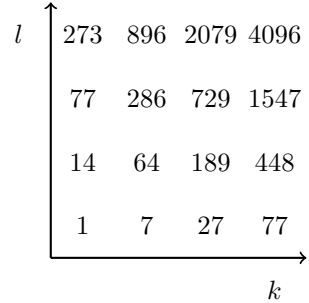

By embedding $\mathfrak{sl}_3$ into $\mathfrak{g}_2$ via the long roots, we can ask how an irreducible representation of $\mathfrak{g}_2$ decomposes when restricted to $\mathfrak{sl}_3$. We can conjecture the decomposition rule, also called the 
\textit{branching rule}, by matching up dimensions, i.e. picking a number $d$ from the right table, and finding a consistent array of numbers from the left table that sums to $d$. 

 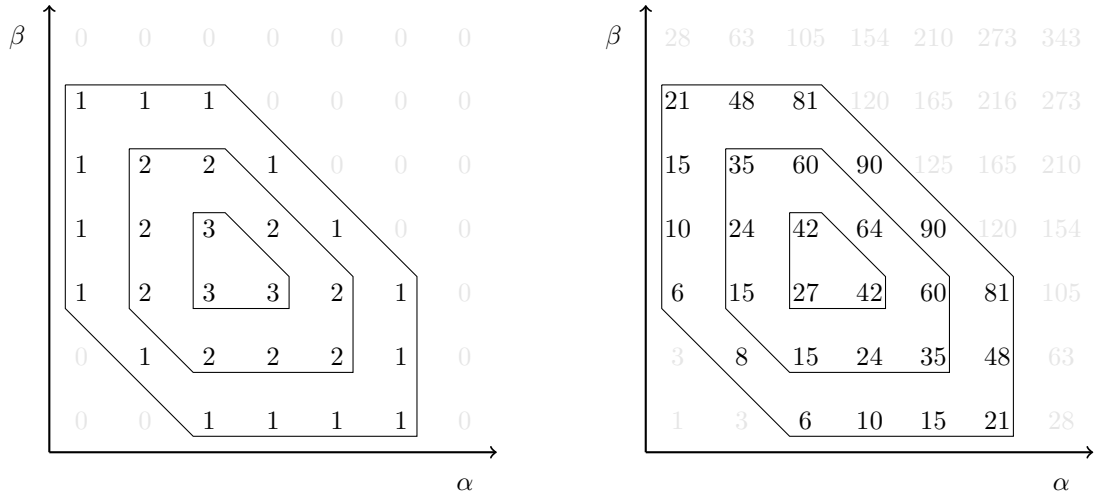
\begin{figure}   
 \centering
 \begin{subfigure}{0.48\textwidth}
   \centering
    \begin{tikzpicture}[scale=0.85]
\draw[thick] (-0.5,-0.5) edge[->] (6.5,-0.5);
\draw[thick] (-0.5,-0.5) edge[->] (-0.5,6.5);
\draw (6, -1) node {$\alpha$};
\draw (-1, 6) node {$\beta$};
\draw (-0.25,5.25)--(2.25,5.25) -- (5.25,2.25) -- (5.25,-0.25) -- (1.75,-0.25) -- (-0.25,1.75) -- (-0.25,5.25);
\draw (0.75,4.25)--(2.25,4.25) -- (4.25,2.25) -- (4.25,0.75) -- (1.75,0.75) -- (0.75,1.75) -- (0.75,4.25);
\draw (1.75,3.25)--(2.25,3.25) -- (3.25,2.25) -- (3.25,1.75) -- (1.75,1.75) -- (1.75,3.25);
\draw[color=gray!20] (0,0) node {$0$};
\draw[color=gray!20] (1,0) node {$0$};
\draw (2,0) node {$1$};
\draw (3,0) node {$1$};
\draw (4,0) node {$1$};
\draw (5,0) node {$1$};
\draw[color=gray!20] (6,0) node {$0$};
\draw[color=gray!20] (0,1) node {$0$};
\draw (1,1) node {$1$};
\draw (2,1) node {$2$};
\draw (3,1) node {$2$};
\draw (4,1) node {$2$};
\draw (5,1) node {$1$};
\draw[color=gray!20] (6,1) node {$0$};
\draw (0,2) node {$1$};
\draw (1,2) node {$2$};
\draw (2,2) node {$3$};
\draw (3,2) node {$3$};
\draw (4,2) node {$2$};
\draw (5,2) node {$1$};
\draw[color=gray!20] (6,2) node {$0$};
\draw (0,3) node {$1$};
\draw (1,3) node {$2$};
\draw (2,3) node {$3$};
\draw (3,3) node {$2$};
\draw (4,3) node {$1$};
\draw[color=gray!20] (5,3) node {$0$};
\draw[color=gray!20] (6,3) node {$0$};
\draw (0,4) node {$1$};
\draw (1,4) node {$2$};
\draw (2,4) node {$2$};
\draw (3,4) node {$1$};
\draw[color=gray!20] (4,4) node {$0$};
\draw[color=gray!20] (5,4) node {$0$};
\draw[color=gray!20] (6,4) node {$0$};
\draw (0,5) node {$1$};
\draw (1,5) node {$1$};
\draw (2,5) node {$1$};
\draw[color=gray!20] (3,5) node {$0$};
\draw[color=gray!20] (4,5) node {$0$};
\draw[color=gray!20] (5,5) node {$0$};
\draw[color=gray!20] (6,5) node {$0$};
\draw[color=gray!20] (0,6) node {$0$};
\draw[color=gray!20] (1,6) node {$0$};
\draw[color=gray!20] (2,6) node {$0$};
\draw[color=gray!20] (3,6) node {$0$};
\draw[color=gray!20] (4,6) node {$0$};
\draw[color=gray!20] (5,6) node {$0$};
\draw[color=gray!20] (6,6) node {$0$};
\end{tikzpicture}
\caption{Multiplicities, indexed by $\alpha$ and $\beta$}
\label{f2a}
\end{subfigure}\hfill\begin{subfigure}{0.48\textwidth}
    \begin{tikzpicture}[scale=0.85]
\draw[thick] (-0.5,-0.5) edge[->] (6.5,-0.5);
\draw[thick] (-0.5,-0.5) edge[->] (-0.5,6.5);
\draw (6, -1) node {$\alpha$};
\draw (-1, 6) node {$\beta$};
\draw (-0.25,5.25)--(2.25,5.25) -- (5.25,2.25) -- (5.25,-0.25) -- (1.75,-0.25) -- (-0.25,1.75) -- (-0.25,5.25);
\draw (0.75,4.25)--(2.25,4.25) -- (4.25,2.25) -- (4.25,0.75) -- (1.75,0.75) -- (0.75,1.75) -- (0.75,4.25);
\draw (1.75,3.25)--(2.25,3.25) -- (3.25,2.25) -- (3.25,1.75) -- (1.75,1.75) -- (1.75,3.25);
\draw[color=gray!20] (0,0) node {$1$};
\draw[color=gray!20] (1,0) node {$3$};
\draw (2,0) node {$6$};
\draw (3,0) node {$10$};
\draw (4,0) node {$15$};
\draw (5,0) node {$21$};
\draw[color=gray!20] (6,0) node {$28$};
\draw[color=gray!20] (0,1) node {$3$};
\draw (1,1) node {$8$};
\draw (2,1) node {$15$};
\draw (3,1) node {$24$};
\draw (4,1) node {$35$};
\draw (5,1) node {$48$};
\draw[color=gray!20] (6,1) node {$63$};
\draw (0,2) node {$6$};
\draw (1,2) node {$15$};
\draw (2,2) node {$27$};
\draw (3,2) node {$42$};
\draw (4,2) node {$60$};
\draw (5,2) node {$81$};
\draw[color=gray!20] (6,2) node {$105$};
\draw (0,3) node {$10$};
\draw (1,3) node {$24$};
\draw (2,3) node {$42$};
\draw (3,3) node {$64$};
\draw (4,3) node {$90$};
\draw[color=gray!20] (5,3) node {$120$};
\draw[color=gray!20] (6,3) node {$154$};
\draw (0,4) node {$15$};
\draw (1,4) node {$35$};
\draw (2,4) node {$60$};
\draw (3,4) node {$90$};
\draw[color=gray!20] (4,4) node {$125$};
\draw[color=gray!20] (5,4) node {$165$};
\draw[color=gray!20] (6,4) node {$210$};
\draw (0,5) node {$21$};
\draw (1,5) node {$48$};
\draw (2,5) node {$81$};
\draw[color=gray!20] (3,5) node {$120$};
\draw[color=gray!20] (4,5) node {$165$};
\draw[color=gray!20] (5,5) node {$216$};
\draw[color=gray!20] (6,5) node {$273$};
\draw[color=gray!20] (0,6) node {$28$};
\draw[color=gray!20] (1,6) node {$63$};
\draw[color=gray!20] (2,6) node {$105$};
\draw[color=gray!20] (3,6) node {$154$};
\draw[color=gray!20] (4,6) node {$210$};
\draw[color=gray!20] (5,6) node {$273$};
\draw[color=gray!20] (6,6) node {$343$};
\end{tikzpicture}
\caption{Hexagons in $A_{\alpha, \beta}$}
\label{f2b}
\end{subfigure}
\caption{Hexagons with pointwise product $G_{3,2}$}
\end{figure}

Note that $G_{k,0}$ is the sum of $A_{\alpha,\beta}$ over the triangle with vertices $A_{(0,0)},A_{(k,0)},A_{(0,k)}$. Similarly, $G_{0,l}$ is the sum of $A_{\alpha,\beta}$ over the triangle with vertices $A_{(l,l)},A_{(l,0)},A_{(0,l)}$. If we look at the nondegenerate example $G_{3,2}=1547$, it is the sum of the pointwise product of the following two hexagons on the $(\alpha,\beta)$-plane, where the second hexagon is a subset of the array of numbers $A_{\alpha,\beta}$.

In other words, $G_{3,2}$ is the weighted sum of $A_{\alpha,\beta}$ on the hexagon with vertices 
    $$A_{5,2},A_{5,0},A_{2,0},A_{0,2},A_{0,5},A_{2,5},$$
where the outer layer is counted with multiplicity one, the middle layer is counted with multiplicity two, and the inner triangular layer is counted with multiplicity three. After some experimentation, we can derive the following rule: 
	$$
    G_{k,l} = \sum_{\alpha,\beta} n_{\alpha,\beta} A_{\alpha,\beta},
    $$
where $(\alpha,\beta)$ are integral points on and inside of the hexagon
	$$
    
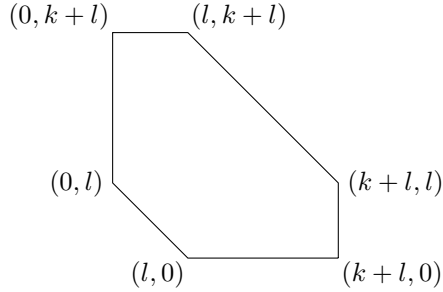
\begin{tikzpicture}
\draw (3,0) -- (1,0) node[below left=-0.08cm] {$(l,0)$}
 -- (0,1) node[left] {$(0,l)$}
      -- (0,3) node[above left=-0.08cm] {$(0,k+l)$}
 -- (1,3) node[above right=-0.08cm] {$(l,k+l)$}
--  (3,1) node[right] {$(k+l,l)$}
           --  (3,0) node[below right=-0.08cm] {$(k+l,0)$};
\end{tikzpicture}
    $$
with vertices
	$$
    (k+l,l),(k+l,0),(l,0),(0,l),(0,k+l),(l,k+l),
    $$
and $n_{\alpha,\beta}$ are positive integers determined as follows.
	\begin{itemize}
    \item If $(\alpha,\beta)$ lies on the perimeter (the zeroth layer) of the hexagon $H$ above, then $n_{\alpha,\beta}=1$.
    \item If $(\alpha,\beta)$ lies on the first layer of $H$ (which are points adjacent to the perimeter), then $n_{\alpha,\beta}=2$.
    \item Iterating, if $(\alpha,\beta)$ lies on the $j^{th}$ layer of $H$, and if this $j^{th}$ layer is still a hexagon, then $n_{\alpha,\beta}=j+1$.
    \item The hexagon $H$ degenerates at the $m^{th} = \min\{k,l\}^{th}$ layer to a triangle with vertices $(l,k),(k,l),(l,l)$ (or possibly the single point $(l,l)$ if $k = l$). Set $n_{\alpha,\beta}=m+1$ for all points $(\alpha,\beta)$ on this triangle.
    \end{itemize}
In Section \ref{g2a2} we will show that this decomposition of $G_{k,l}$ into $A_{\alpha,\beta}$ works on the representation theoretic level as well.

We now raise a few questions about the branching rule of the restriction functor on simple Lie algebras.
    \begin{qn}
        Given an embedding of a simple Lie algebra $\mathfrak{a}$ into $\mathfrak{g}$, can we give an explicit branching rule for $\Rest ^{\mathfrak{g}}_{\mathfrak{a}}$ like the one for $\Rest ^{\mathfrak{g}_2}_{\mathfrak{sl}_3}$ above?
    \end{qn} 
    \begin{qn}
    What governs the coefficients of the branching rule? For example, the coefficients for $\Rest ^{\mathfrak{g}_2}_{\mathfrak{sl}_3}$ is the weighted hexagon illustrated above.
    \end{qn} 
     \begin{qn}
        How many irreducible factors of $\mathfrak{a}$ are there in $\Rest ^{\mathfrak{g}}_{\mathfrak{a}} \Pi_{\lambda}$, where $\Pi_{\lambda}$ is an irreducible representation of $\mathfrak{g}$? In particular, what is the sum of the coefficients of the branching rule?
    \end{qn} 
 In this paper we will work with splint root systems. Then Question 1 is related to Weyl group symmetric functions and the Littlewood-Richardson rule if viewed combinatorially, and Question 2 is related to the weight diagram of a sub-root system corresponding to the splint root system. A solution to Question 3 falls out from a satisfactory answer to Question 1, and is related to the dimension of a particular irreducible representation of an auxiliary Lie algebra. For example, in the above case with $\mathfrak{g}_2$ and $\mathfrak{sl}_3$ we have the curious identity
 $$
 \sum_{\alpha,\beta}n_{\alpha,\beta}= A_{k,l}.
 $$

\section{Splint root systems} \label{srssec}

Let $\Delta$ be a simple root system. We want to study the root systems for which $\Delta$ is \textit{splint}, i.e. $\Delta=\Delta_1\sqcup \Delta_2$ is a disjoint union of two root systems $\Delta_1$ and $\Delta_2$, each of which is embedded into $\Delta$ as an additive group, with $\Delta_1$ embedded metrically and $\Delta_2$ embedded in such a way that the length of roots are scaled uniformly. The notion of a splint was introduced by David Richter in \cite{5}, and he gave a classification of possible splints of root systems (including the cases where $\Delta_1$ may not be embedded metrically, for which we do not consider). The table below lists all possible splint root systems, and we label them Types (I) to (V).
    $$
\begin{tabular}{ c | c | c | c }
 Type & $\Delta$ & $\Delta_1$ & $\Delta_2$ \\ \hline
 (I) & $A_r$ $(r\geq 2)$ & $A_{r-1}$ & $(A_1)^{\oplus r}$ \\
 (II) & $B_r$ $(r\geq 2)$ & $D_r$ & $(A_1)^{\oplus r}$ \\
 (III) & $C_r$ $(r\geq 3)$ & $(A_1)^{\oplus r}$ & $D_r$ \\
 (IV) & $G_2$ & $A_2$ & $A_2$ \\
 (V) & $F_4$ & $D_4$ & $D_4$ \\
\end{tabular}
    $$
We note that the last four types of splint root system have $\Delta_2$ embedded metrically into $\Delta$.

Now write $\mathfrak{a}$ to be the Lie algebra of $\Delta_1$, corresponding to a Lie subalgebra of $\g$. Letting $\Pi_\lambda$ be an irreducible representation of $\g$ of highest weight $\lambda$, we have a decomposition
$$
\text{Res}^{\g}_{\mathfrak{a}}\Pi_{\lambda} = \bigoplus_{\nu} b_{\lambda,\nu} \pi_{\nu}
$$
We are interested in computing the branching coefficients $b_{\lambda,\nu}$. The branching coefficients for Types (I) and (II) are well known examples of Gelfand-Tsetlin patterns \cite{GT1}, which we now state. For Type (I), every irreducible representation of $\ssl_{r+1}$ is indexed by a Young tableau $Y$ with at most $r$ rows, and its restriction to $\ssl_{r}$ is the direct sum of irreducible representations of $\ssl_{r}$ corresponding to those Young tableaux obtained from $Y$ by removing some boxes, each of multiplicity one. Explicitly, if $\pi^r_{\lambda_1,\dots,\lambda_r}$ is a highest weight representation of $\ssl_{r+1}$ with $\lambda_i\geq\lambda_{i+1}$, then
    $$
    \Rea \pi^{r}_{\lambda_1,\dots,\lambda_{r}} = 
    \bigoplus_{\lambda_i\geq \mu_i\geq \lambda_{i+1}} \pi^{r-1}_{\mu_1,\dots,\mu_{r-1}}.
    $$
As for Type (II), recall that every irreducible representation of $\so_{N}$ is labeled by 
    \begin{itemize}
    \item $f_1\geq \dotsm\geq f_{r-1}\geq f_r\geq 0$ if $N=2r+1$,
    \item $f_1\geq \dotsm\geq f_{r-1}\geq |f_r|$ if $N=2r$.
    \end{itemize}
where the $f_i$'s are simultaneously integers or half-integers. If we write $\Pi^r_{f_1,\dots,f_r}$ to be the highest weight representation of $\so_{2r+1}$, and if we write $\pi^r_{g_1,\dots,g_r}$ to be the highest weight representation of $\so_{2r}$, then the branching rule is
    \begin{align*}
    \Rest^{\so_{2r+1}}_{\so_{2r}} \Pi^r_{f_1,\dots,f_r} &= \bigoplus_{\substack{f_1\geq g_1\geq f_2\geq\dotsm\geq f_{r-1}\geq g_{r-1}\geq f_r\geq |g_r| \\ f_i-g_i\in\mathbb{Z}}} \pi^r_{g_1,\dots,g_r}.
    \end{align*}
We would like similar explicit branching rules for the other three types of splint root system listed above. 

A computationally intensive heuristic for the branching coefficients $b_{\lambda,\nu}$ exists in \cite{3}. In this heuristic, the computation of $b_{\lambda,\nu}$ relies on the roots $\Delta\setminus\Delta_1$.
    \begin{thm}[{\cite[Property 2.1]{3}}]
    Let $m_{\Delta_2,\mu,\nu}$ be the multiplicity of $\tilde{\nu}$ from the weight diagram of $\Delta_2$ with highest weight $\tilde{\mu}$. Then 
        $$
        m_{\Delta_2,\mu,\nu}=b_{\mu-\phi(\tilde{\mu}-\tilde{\nu})},
        $$
    where $\phi$ is the embedding of $\Delta_2$ into $\Delta$.
    \end{thm}
This theorem, together with Freudenthal's Multiplicity Formula \cite[Section 22]{hum} tells us all the branching coefficients in principle. However, this is not easy to compute in practice. Our goal in this paper is to give a framework to understand the branching coefficients directly using the Weyl character formula and give an explicit formula for the Type (IV) branching rule, as well as conjecture formulas for Types (III) and (V) branching rules.

\section{Preliminaries}

\subsection{The Weyl character and dimension formulas}

We recall some computational tools from representation theory. These are very classical results (see \cite{1} for an exposition, for instance), and our main purpose is to fix notation.

Let $G$ be a compact simply connected Lie group, and let $T$ be a maximal torus of $G$. Then the Lie algebra $\g$ of $G$ can be written as
	$$
    \g = \lt \oplus \lp,
    $$
where $\lt=\Lie(T)$ and $\lp=\Lie(G/T)$ is the subspace of eigenvectors for the roots.

For any irreducible representation $L_{\lambda}$ of $G$ with highest weight $\lambda$, we can decompose $L_{\lambda}$ into its weight decomposition
	$$
    L_{\lambda} = \bigoplus_{\mu\in T^*} L_{\lambda}[\mu]
    $$
 where $L_{\lambda}[\mu] = \{v\in L_{\lambda} : tv = \mu(t) v \text{ for all $t\in T$}\}$. Define its \textit{character} to be the finite sum
	$$
    \chi(L_{\lambda}) = \sum_{\mu\in T^*} \dim (L_{\lambda}[\mu]) e^{\mu}.
    $$

\begin{thm}[Weyl Character Formula]
Let $W$ be the Weyl group of $G$, and let $l(w)$ be the length of an element $w\in W$. Then
	$$
    \chi(L_{\lambda}) = \frac{\sum_{w\in W}(-1)^{l(w)}e^{w(\lambda+\rho)}}{\delta},
    $$
where 
	$$
    \delta = e^{\rho} \prod_{\alpha\in R^+}(1-e^{-\alpha}) = \prod_{\alpha\in R^+}(e^{\alpha/2}-e^{-\alpha/2}),
    $$
and $\rho$ is the half-sum of the positive roots $R^+$.
\end{thm}

The formula below allows us to compute the dimension of any irreducible representation of $G$.

\begin{thm}[Weyl Dimension Formula]
Let $L_{\lambda}$ be the irreducible representation of $G$ with highest weight $\lambda$. Then
	$$
    \dim (L_{\lambda}) = \prod_{\alpha\in R^+} \frac{(\lambda+\rho,\alpha)}{(\rho,\alpha)}.
    $$
\end{thm}

\subsection{A strategy} \label{strat}

Let us return to the notations introduced in Section 1. Our strategy to write down explicit branching coefficient $b_{\lambda,\nu}$ is as follows. We check that our branching coefficients are plausible by first verifying that $\dim \text{Res}^{\g}_{\mathfrak{a}}\Pi_{\lambda}$ and $\dim \bigoplus_{\nu} b_{\lambda,\nu} \pi_{\nu}$ agree. Then we will use the Weyl character formula to make sure that the weight multiplicities check out. A way to do this is as follows. Write the denominator $\delta_{\mathfrak{g}}$ of the Weyl character formula for $\Pi_{\lambda}$ as 
    $$
    \delta_{\mathfrak{g}} = \delta_{\mathfrak{a}}\delta',
    $$
where $\delta'$ corresponds to the roots of $\mathfrak{g}$ inside $\mathfrak{g}\setminus\mathfrak{a}$. Observe that the function $\delta' \chi(\Pi_{\lambda})$ is a symmetric function on the Weyl group $W_{\mathfrak{a}}$ of the root system of $\mathfrak{a}$. Hence we can write both $\delta' \chi(\Pi_{\lambda})$ and $\delta'$ as a polynomial in $\chi(\pi_\mu)$ and compute branching coefficients by comparing 
    $$
    \delta' \chi(\Pi_{\lambda}) \qquad\text{and}\qquad
    \delta' \chi(\pi_\mu).
    $$
If $W_{\mathfrak{a}}$ is the symmetric group, then the latter product can be understood using the Littlewood-Richardson rule; we will see this when we prove the Type (IV) branching rule in Section \ref{g2a2}. In general one would need to employ a suitable Littlewood-Richardson rule for $W_{\mathfrak{a}}$. 

In our computations we are led to the following conjecture.

\begin{conj} \label{conjj}
Write $\Rest^{\g}_{\mathfrak{a}}\Pi_{\lambda}=\bigoplus_{\nu} b_{\lambda,\nu} \pi_{\nu}$. Then
$$\sum_{\lambda} b_{\lambda,\mu} = \dim \omega_{\lambda},$$
where $\omega_{\lambda}$ is a highest weight representation (depending on $\lambda$) for the root system $\Xi$ of an auxiliary simple Lie algebra.
\end{conj}

The Gelfand-Tsetlin patterns for Types (I) and (II) imply that $\Xi$ can be taken to be $\Delta_2$. For example, if we index the irreducible representations of $B_3$ and $D_3$ by three positive integers after choosing the standard fundamental weights, then the Gelfand-Tsetlin pattern for $\Resss$ can be written as
	$$
    \Resss \Pi_{a,b,c} = \bigoplus_{t=0}^a\bigoplus_{r=0}^b\bigoplus_{s=0}^c \pi_{a+b-t-r,r+s,r+c-s}.
    $$
In this case, the sum of coefficients equals $\dim \omega_{a,b,c}$, where $\omega_{a,b,c}$ is the highest weight representation of $A_1^{\oplus 3}$ corresponding to the integers $a,b,c$.

The branching rule for Type (IV) proven in the next section will imply that $\Xi=\Delta_2$ as well, and we conjecture this is also the case for Type (III). However, the discussion in Section \ref{here} tells us this is not the case for Type (V).

\section{Branching rule for Type (IV)}

In this section we work out $\Res$ explicitly. We first give an explicit formula for the functor $\Ress$ without using Gelfand-Tsetlin patterns in order to illustrate the ideas used in understanding $\Res$.

\subsection{Branching rule for Type (II) with \texorpdfstring{$r=2$}{}}

As $D_2$ embeds into $B_2$ via the long roots, it is natural to ask how their irreducible representations are related. The starting point is to compute their Weyl character formulas. To do this we label roots $L_1,L_2$, and all the positive roots, as below.
$$
  
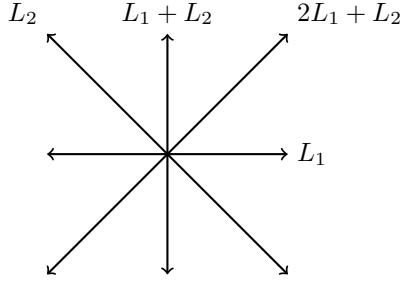
\begin{tikzpicture}[scale=0.8]
    \foreach\ang in {0,90,180,270}{
     \draw[->,thick] (0,0) -- (\ang:2cm);
    }
    \foreach\ang in {45,135,225,315}{
     \draw[->,thick] (0,0) -- (\ang:2.828cm);
    }
    \node[anchor=west,scale=1] at (2,0) {$L_1$};
    \node[anchor=south,scale=1] at ({2*cos(90)},{2*sin(90)}) {$L_1+L_2$};
    \node[anchor=south west,scale=1] at ({2.828*cos(45)},{2.828*sin(45)}) {$2L_1+L_2$};
    \node[anchor=south east,scale=1] at ({2.828*cos(135)},{2.828*sin(135)}) {$L_2$};
  \end{tikzpicture}
$$
The fundamental weights $\omega_1,\omega_2$ and $\Omega_1,\Omega_2$ for $B_2$ and $D_2$ are
	$$
    \omega_1=L_1+L_2,\qquad \omega_2=\frac{2L_1+L_2}{2},\qquad \Omega_1=\frac{L_2}{2},\qquad \Omega_2=\frac{2L_1+L_2}{2},
    $$
and the half sum of the positive roots for $B_2$ and $D_2$ are
	$$
    \rho_{B_2} = 2L_1+\frac{3L_2}{2},\qquad \rho_{D_2}=L_1+L_2.
    $$
Define $\Pi_{k,l}$ to be the highest weight representation of $B_2$ with weight $k\omega_1+l\omega_2=(2k+2l)L_1/2+(2k+l)L_2/2$, and define $\pi_{\alpha,\beta}$ to be the highest weight representation of $D_2$ with weight $\alpha\Omega_1+\beta\Omega_2=\beta L_1 + (\alpha + \beta) L_2/2$. By writing $x_1=e^{L_1+L_2/2}$ and $x_2=e^{L_2/2}$, we have the following explicit formulas for the characters of $\Pi_{k,l}$ and $\pi_{\alpha,\beta}$:
	$$
    \chi(\Pi_{k,l}) = 
    \frac{A_{k,l,B_2}}{\delta_{B_2}},\qquad
    \chi(\pi_{\alpha,\beta})=\frac{A_{\alpha,\beta,D_2}}{\delta_{D_2}},
    $$
where
\begin{align*}
    A_{k,l,B_2} &= 
    x_1^{k+l+2}x_2^{k+1}
    + x_2^{k+l+2}x_1^{-(k+1)}
    + x_1^{-(k+l+2)}x_2^{-(k+1)}
    + x_2^{-(k+l+2)}x_1^{k+1}\\
    &\qquad -x_2^{-(k+l+2)}x_1^{-(k+1)}
    - x_1^{k+l+2}x_2^{-(k+1)}
    - x_2^{k+l+2}x_1^{k+1}
    - x_1^{-(k+l+2)}x_2^{k+1}, \\
    \delta_{B_2} &= (x_1-x_1^{-1})(x_2-x_2^{-1})(x_1+x_1^{-1}-x_2-x_2^{-1}), \\
    A_{\alpha,\beta,D_2} &= 
    (x_2^{\alpha+1}-x_2^{-(\alpha+1)})(x_1^{\beta+1}-x_1^{-(\beta+1)}), \\
    \delta_{D_2} &= (x_1-x_1^{-1})(x_2-x_2^{-1}).
    \end{align*}
Finally let us write down the Weyl dimension formulas for $B_2$ and $D_2$:
	\begin{align*}
    \dim \Pi_{k,l} &= \frac{(k+1)(l+1)(k+l+2)(2k+l+3)}{6} \\
    \dim \pi_{\alpha,\beta} &= (\alpha+1)(\beta+1)
    \end{align*}

\begin{prop} \label{ressthm}
We have
	$$
    \Ress \Pi_{k,l} = \bigoplus_{r=0}^k\bigoplus_{s=0}^l \pi_{r+s,r+l-s}.
    $$
\end{prop}

\begin{cor}
The number of irreducible representations of $D_2$ in the decomposition of $\Ress \Pi_{k,l}$ equals 
	$$
    \dim \pi_{k,l} = (k+1)(l+1).
    $$
\end{cor}

This corollary is an immediate consequence of Proposition \ref{ressthm}, so we just need to prove the above theorem. For this case we can simply use a telescoping sum argument to compute the Weyl character formula on both sides, but in general we want approaches that will allow us to deduce the decomposition from our computations. To this end we give two approaches to the proof: the first approach is bare-hands computation, and the second approach is an explicit computation using the strategy described in the previous section. 

\begin{proof}[Proof 1]
Factor $A_{k,l,B_2}$ as
    $$
    A_{k,l,B_2} = 
    (x_2^{k+1}-x_2^{-(k+1)})(x_1^{k+l+2}-x_1^{-(k+l+2)})-
    (x_1^{k+1}-x_1^{-(k+1)})(x_2^{k+l+2}-x_2^{-(k+l+2)}).
    $$
Then we observe that
\begin{align*}
    \frac{A_{k,l,B_2}}{(x_1-x_1^{-1})(x_2-x_2^{-1})}
    &=
    (x_2^{k}+x_2^{k-2}+\dotsm+x_2^{-k})(x_1^{k+l+1}+x_1^{k+l-1}+\dotsm+x_1^{-(k+l+1)}) \\
     &\qquad -
        (x_1^{k}+x_1^{k-2}+\dotsm+x_1^{-k})(x_2^{k+l+1}+x_2^{k+l-1}+\dotsm+x_2^{-(k+l+1)}).
\end{align*}
We can view the above expressions as sums over the polynomial
    $$
    p(s,t)=x_2^{s}x_1^{t}-x_1^{s}x_2^{t}+x_2^{-s}x_1^{-t}-x_1^{-s}x_2^{-t},
    $$
where $s$ ranges over the nonnegative numbers in $\{k,k-2,\dots,-k\}$ and $t$ ranges over the nonnegative numbers in $\{k+l+1,k+l-1,\dots,-(k+l+1)\}$. Writing $u=x_1^{1/2}x_2^{1/2}$ and $v=x_1^{1/2}x_2^{-1/2}$, we can write
    \begin{align*}
    p(s,t) &= u^{s+t}v^{-s+t}-u^{s+t}v^{s-t}+u^{-s-t}v^{s-t}-u^{-s-t}v^{-s+t} \\
    &= (u^{s+t}-u^{-(s+t)})(v^{-s+t}-v^{-(-s+t)}).
    \end{align*}
Now, after observing $x_1+x_1^{-1}-x_2-x_2^{-1}=(u-u^{-1})(v-v^{-1})$, we get
\begin{align*}
    \frac{p(s,t)}{(u-u^{-1})(v-v^{-1})}
    &=
    (u^{s+t-1}+u^{s+t-3}+\dotsm+u^{-(s+t-1)})(v^{-s+t-1}+v^{-s+t-3}+\dotsm+v^{-(-s+t-1)}).
\end{align*}
By comparing this expression with the Weyl character formula for $D_2$
    \begin{align*}
    \frac{A_{\alpha,\beta,D_2}}{\delta_{D_2}}
    &=
    (x_2^{\alpha}+x_2^{\alpha-2}+\dotsm+x_2^{-\alpha})(x_1^{\beta}+x_1^{\beta-2}+\dotsm+x_1^{-\beta}) \\
    &= \sum_{\substack{a\in \{\alpha,\alpha-2,\dots,-\alpha\} \\ b\in \{\beta,\beta-2,\dots,-\beta\}}} u^{a+b}v^{-a+b},
    \end{align*}
we get what we want.
\end{proof}

\begin{proof}[Proof 2]
Write $\chi_{\alpha,\beta}=\chi(\pi_{\alpha,\beta})$. By factoring $A_{k,l,B_2}$ as above, we observe that
    $$
    \chi(\Pi_{k,l}) = \frac{\chi_{k,k+l+1}-\chi_{k+l+1,k}}{\chi_{0,1}-\chi_{1,0}}.
    $$
Note that
    $$
    (\chi_{0,1}-\chi_{1,0})\chi_{\alpha,\beta}
    =
    \chi_{\alpha,\beta+1}+\chi_{\alpha,\beta-1}-\chi_{\alpha+1,\beta}-\chi_{\alpha-1,\beta},
    $$
where the second term exists only when $\beta>0$, and the last term exists only when $\alpha>0$. On the $(\alpha,\beta)$-plane, this amounts to taking the weighted sum of the following four vertices, with sign as below.
    $$
    
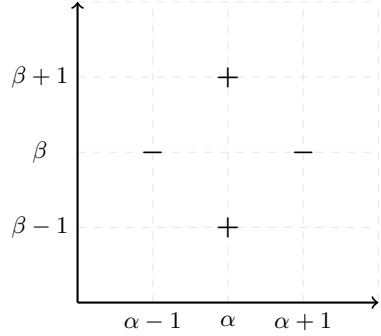
\begin{tikzpicture}
\draw[help lines, color=gray!20, dashed] (-1,-1) grid (3,3);
\draw[->,thick] (-1,-1)--(3,-1);
\draw[->,thick] (-1,-1)--(-1,3);
\draw (0,-1.25) node {\text{\small $\alpha-1$}};
\draw (1,-1.25) node {\text{\small $\alpha$}};
\draw (2,-1.25) node {\text{\small $\alpha+1$}};
\draw (-1.5,0) node {\text{\small $\beta-1$}};
\draw (-1.5,1) node {\text{\small $\beta$}};
\draw (-1.5,2) node {\text{\small $\beta+1$}};
\draw (2,1) node {$\bm{-}$};
\draw (0,1) node {$\bm{-}$};
\draw (1,2) node {$\bm{+}$};
\draw (1,0) node {$\bm{+}$};
\end{tikzpicture}
    $$
We can now easily check that
    $$
    \chi_{k,k+l+1}-\chi_{k+l+1,k}
    =
    \sum_{\alpha=r}^k \sum_{\beta=s}^l
    (\chi_{0,1}-\chi_{1,0})\chi_{r+s,r+l-s},
    $$
as desired.
\end{proof}

\subsection{Branching rule for Type (IV)} \label{g2a2}

Again $A_2$ embeds into $G_2$ via the long roots. We need to compute the Weyl character formula for $G_2$ and $A_2$. To do this we label roots $L_1,L_2,L_3$, and all the positive roots, as below.
$$
  
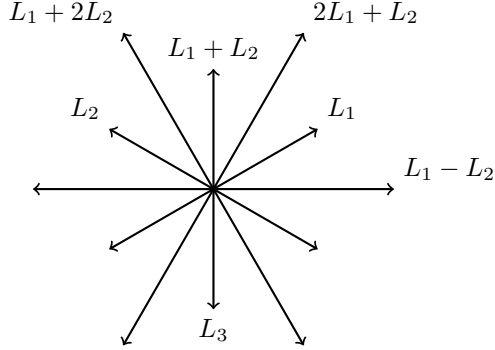
\begin{tikzpicture}[scale=0.8]
    \foreach\ang in {30,90,...,330}{
     \draw[->,thick] (0,0) -- (\ang:2cm);
    }
    \foreach\ang in {0,60,...,300}{
     \draw[->,thick] (0,0) -- (\ang:3cm);
    }
    \node[anchor=south west,scale=1] at (3,0) {$L_1-L_2$};
    \node[anchor=south west,scale=1] at ({3*cos(60)},{3*sin(60)}) {$2L_1+L_2$};
    \node[anchor=south east,scale=1] at ({3*cos(120)},{3*sin(120)}) {$L_1+2L_2$};
    \node[anchor=south west,scale=1] at ({2*cos(30)},{2*sin(30)}) {$L_1$};
    \node[anchor=south,scale=1] at ({2*cos(90)},{2*sin(90)}) {$L_1+L_2$};
    \node[anchor=south east,scale=1] at ({2*cos(150)},{2*sin(150)}) {$L_2$};
    \node[anchor=north,scale=1] at ({2*cos(90)},{-2*sin(90)}) {$L_3$};
  \end{tikzpicture}
$$
We chose the labeling above because the action of the Weyl group $W_{A_2}\cong S_3$ on $L_1,L_2,L_3$ is simply by permuting the indices. The fundamental weights $\omega_1,\omega_2$ and $\Omega_1,\Omega_2$ for $G_2$ and $A_2$ are
	$$
    \omega_1=L_1+L_2,\qquad \omega_2=2L_1+L_2,\qquad \Omega_1=L_1+L_2,\qquad \Omega_2=L_1,
    $$
and the half sum of the positive roots are
	$$
    \rho_{G_2} = 3L_1+2L_2,\qquad \rho_{A_2}=2L_1+L_2.
    $$
Define $\Pi_{k,l}$ to be the highest weight representation of $G_2$ with weight $k\omega_1+l\omega_2=(k+2l)L_1+(k+l)L_2$, and define $\pi_{\alpha,\beta}$ to be the highest weight representation of $A_2$ with weight $\alpha\Omega_1+\beta\Omega_2=(\alpha+\beta)L_1+\alpha L_2$. By writing $x_i=e^{L_i}$, we have the following explicit formulas for the characters of $\Pi_{k,l}$ and $\pi_{\alpha,\beta}$:
	$$
    \chi(\Pi_{k,l}) = 
    \frac{A_{k,l,\mg}}{\delta_{\mg}},\qquad
    \chi(\pi_{\alpha,\beta})=\frac{A_{\alpha,\beta,\su(3)}}{\delta_{\su(3)}},
    $$
where
\begin{align*}
    A_{k,l,G_2} &= 
    x_1^{k+2l+3}x_2^{k+l+2}
    + x_3^{k+2l+3}x_1^{k+l+2}
    + x_2^{k+2l+3}x_3^{k+l+2} \\
    &\qquad+ x_1^{-(k+2l+3)}x_2^{-(k+l+2)}
    + x_2^{-(k+2l+3)}x_3^{-(k+l+2)}
    + x_3^{-(k+2l+3)}x_1^{-(k+l+2)} \\
    &\qquad - x_1^{k+2l+3}x_3^{k+l+2}
    - x_3^{k+2l+3}x_2^{k+l+2}
    - x_2^{k+2l+3}x_1^{k+l+2} \\
    &\qquad -x_1^{-(k+2l+3)}x_3^{-(k+l+2)}
    - x_2^{-(k+2l+3)}x_1^{-(k+l+2)}
    - x_3^{-(k+2l+3)}x_2^{-(k+l+2)},\\
    \delta_{G_2} &= (x_1-x_2)(x_1-x_3)(x_2-x_3)(1-x_1)(1-x_2)(1-x_3), \\
    A_{\alpha,\beta,A_2} &= 
    x_1^{\alpha+\beta+2}x_2^{\alpha+1}
    + x_3^{\alpha+\beta+2}x_1^{\alpha+1}
    + x_2^{\alpha+\beta+2}x_3^{\alpha+1} \\
    &\qquad -x_1^{\alpha+\beta+2}x_3^{\alpha+1}
    - x_3^{\alpha+\beta+2}x_2^{\alpha+1}
    - x_2^{\alpha+\beta+2}x_1^{\alpha+1}, \\
    \delta_{A_2} &= (x_1-x_2)(x_1-x_3)(x_2-x_3),
    \end{align*}
and $x_1,x_2,x_3$ satisfies the relation $x_1x_2x_3=1$.

\begin{exmp}
The Weyl character formula $\chi(\Pi_{0,1})$ for the adjoint representation $\Pi_{0,1}$ of $G_2$ is
	$$
	\frac{x_1^5x_2^3+ x_3^5x_1^3+ x_2^5x_3^3+x_1^{-5}x_2^{-3}+ x_2^{-5}x_3^{-3}+ x_3^{-5}x_1^{-3}
    -x_1^5x_3^3- x_3^5x_2^3- x_2^5x_1^3-x_1^{-5}x_3^{-3}- x_2^{-5}x_1^{-3}- x_3^{-5}x_2^{-3}}{(x_1-x_2)(x_1-x_3)(x_2-x_3)(1-x_1)(1-x_2)(1-x_3)}.
    $$
We can check using the relation $x_1x_2x_3=1$ that the above expression equals the polynomial
	$$
    x^2y+y^2z+z^2x+xy^2+yz^2+zx^2+xy+yz+zx+x+y+z+2.
    $$
Notice that the above polynomial equals $\chi(\pi_{0,1})+\chi(\pi_{1,0})+\chi(\pi_{1,1})$, so we see that
	$$
    \Res\Pi_{0,1}=\pi_{0,1}\oplus \pi_{1,0}\oplus \pi_{1,1}.
    $$
\end{exmp}

Finally let us write down the Weyl dimension formulas for $G_2$ and $A_2$.
	\begin{align*}
    \dim \Pi_{k,l} &= \frac{(k+1)(k+l+2)(2k+3l+5)(k+2l+3)(k+3l+4)(l+1)}{120} \\
    \dim \pi_{\alpha,\beta} &= \frac{(\alpha+1)(\beta+1)(\alpha+\beta+2)}{2}
    \end{align*}

\begin{thm} \label{g2a2thm}
Let $\Pi_{k,l}$ be the irreducible representation of $G_2$ with weight $k\omega_1+l\omega_2$, and let $\pi_{\alpha,\beta}$ be the irreducible representation of $A_2$ with weight $\alpha\Omega_1+\beta\Omega_2=(\alpha+\beta)L_1+\alpha L_2$. Then
	$$
    \Res \Pi_{k,l} = \bigoplus_{\alpha,\beta} n_{\alpha,\beta} \pi_{\alpha,\beta},
    $$
where $(\alpha,\beta)$ are integral points on and inside of the hexagon
	$$
    
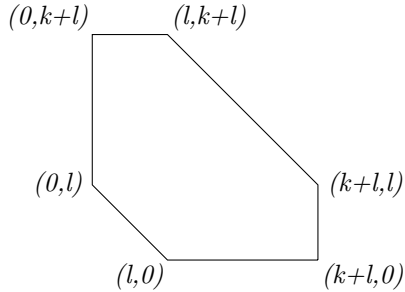
\begin{tikzpicture}
\draw (3,0) -- (1,0) node[below left=-0.08cm] {(l,0)}
 -- (0,1) node[left] {(0,l)}
      -- (0,3) node[above left=-0.08cm] {(0,k+l)}
 -- (1,3) node[above right=-0.08cm] {(l,k+l)}
--  (3,1) node[right] {(k+l,l)}
           --  (3,0) node[below right=-0.08cm] {(k+l,0)};
\end{tikzpicture}
    $$
with vertices
	$$
    (k+l,l),(k+l,0),(l,0),(0,l),(0,k+l),(l,k+l),
    $$
and $n_{\alpha,\beta}$ are positive integers determined as follows.
	\begin{itemize}
    \item If $(\alpha,\beta)$ lies on the perimeter (the zeroth layer) of the hexagon $H$ above, then $n_{\alpha,\beta}=1$.
    \item If $(\alpha,\beta)$ lies on the first layer of $H$ (which are points adjacent to the perimeter), then $n_{\alpha,\beta}=2$.
    \item Iterating, if $(\alpha,\beta)$ lies on the $j^{th}$ layer of $H$, and if this $j^{th}$ layer is still a hexagon, then $n_{\alpha,\beta}=j+1$.
    \item The hexagon $H$ degenerates at the $m^{th} = min(k,l)^{th}$ layer to a triangle with vertices $(l,k),(k,l),(l,l)$ (or possibly the single point $(l,l)$ if $k = l$). Set $n_{\alpha,\beta}=m+1$ for all points $(\alpha,\beta)$ on this triangle.
    \end{itemize}
\end{thm}

\begin{cor}
The number of irreducible representations of $A_2$ in the decomposition of $\Res \Pi_{k,l}$ equals 
	$$
    \dim \pi_{k,l} = \frac{(k+1)(l+1)(k+l+2)}{2}.
    $$
\end{cor}

In the remainder of this section we prove Theorem \ref{g2a2thm}. Again, it is enough to show that the characters of the two sides are equal.  In the numerator $A_{k,l,G_2}$ of the character formula for $G_2$, by separating the terms with positive exponents from those with negative exponents, we may check that  
    $$
    \chi(\Pi_{k,l})= \frac{
    \begin{vmatrix} 
    x_1^{k+2l+3} & x_2^{k+2l+3} & x_3^{k+2l+3} \\
    x_1^{k+l+2} & x_2^{k+l+2} & x_3^{k+l+2} \\
    1 & 1 & 1
    \end{vmatrix}
    +
    \begin{vmatrix}
    x_1^{-(k+2l+3)} & x_2^{-(k+2l+3)} & x_3^{-(k+2l+3)} \\
    x_1^{-(k+l+2)} & x_2^{-(k+l+2)} & x_3^{-(k+l+2)} \\
    1 & 1 & 1
    \end{vmatrix}
    }{(x_1-x_2)(x_1-x_3)(x_2-x_3)(1-x_1)(1-x_2)(1-x_3)}.
    $$
    
One might recognize now in $\chi(\Pi_{k,l})$ something resembling the well-known determinant-based definition of Schur functions on three variables:
    $$
    s_{a_1,a_2,a_3}(x_1,x_2,x_3) = \frac{
    \begin{vmatrix}
    x_1^{a_1+2} & x_2^{a_1+2} & x_3^{a_1+2} \\
    x_1^{a_2+1} & x_2^{a_2+1} & x_3^{a_2+1} \\
    x_1^{a_3} & x_2^{a_3} & x_3^{a_3} 
    \end{vmatrix}
    }{(x_1-x_2)(x_1-x_3)(x_2-x_3)}.
    $$

Clearly the first summand in the numerator combines with the first three factors in the denominator to make $s_{k+2l+1,k+l+1,0}(x_1,x_2,x_3)$.  To simplify the other summand, we use the fact that $x_1x_2x_3 = 1$ to write
    $$
    (x_1-x_2)(x_1-x_3)(x_2-x_3) = 
    -(x_1^{-1}-x_2^{-1})(x_1^{-1}-x_3^{-1})(x_2^{-1}-x_3^{-1})
    $$
and so we can now recognize the full equation as 
    $$
    \chi(\Pi_{k,l}) = \frac{s_{k+2l+1,k+l+1,0}(x_1,x_2,x_3) - s_{k+2l+1,k+l+1,0}(x_1^{-1},x_2^{-1},x_3^{-1})}
    {s_{1,1,0}(x_1,x_2,x_3)-s_{1,0,0}(x_1,x_2,x_3)}.
    $$
One can furthermore eliminate the term with negative exponents. Using \cite[Chapter 7, Exercise 41]{6}, when $x_1x_2x_3 = 1$, we have 
    $$
    s_{k+2l+1,k+l+1,0}(x_1^{-1},x_2^{-1},x_3^{-1}) = s_{k+2l+1,l,0}(x_1,x_2,x_3).
    $$
One can also rewrite the Weyl character formula of $A_2$ as a Schur function in the same manner. To summarize (suppressing the variables now that all have positive exponents):
    \begin{align*}
    \chi(\Pi_{k,l}) &= \frac{s_{k+2l+1,k+l+1,0}-s_{k+2l+1,k,0}}{s_{1,1,0}-s_{1,0,0}}, \\ 
    \chi(\pi_{\alpha,\beta}) &= s_{\alpha+\beta,\alpha,0}.
    \end{align*}
Thus, we must show that 
    $$
    s_{k+2l+1,k+l+1,0} - s_{k+2l+1,k,0} = \sum n_{\alpha,\beta}(s_{1,1,0}-s_{1,0,0})s_{\alpha+\beta,\alpha,0},
    $$
where the sum runs over those $(\alpha,\beta)$ described in the statement of Theorem \ref{g2a2thm}.

In order to simplify the Schur functions that will appear in further computations we also note the following lemma, which follows immediately from the determinant-based definition of Schur functions.

\begin{lem}

In the case that $x_1x_2x_3 = 1$, whenever $\alpha \geq \beta \geq \gamma$ are positive integers, we have
    $$
    s_{\alpha,\beta,\gamma}(x_1,x_2,x_3) = s_{\alpha-\gamma, \beta-\gamma, 0}(x_1,x_2,x_3)
    $$

\end{lem}

Now, we expand the right-hand side of the equation by Pieri's Rule \cite[Chapter 7.15]{6}, which states that, for any partition $\mu$,
    $$
    s_{\mu}s_{1^k} = \sum s_{\lambda},
    $$
where the sum is taken over all partitions $\lambda$ whose Young diagram is formed from the Young diagram of $\mu$ by adding $k$ boxes into $k$ distinct rows.  We thus see (after using the previous lemma to simplify) that
    $$
    s_{1,1,0}s_{\alpha+\beta,\alpha,0} = s_{\alpha+\beta+1, \alpha+1,0} + s_{\alpha+\beta,\alpha-1, 0} + s_{\alpha+\beta-1,\alpha,0},
    $$
where the second summand does not exist if $\alpha = 0$, and the third summand does not exist if $\beta = 0$. Similarly we have
    $$
    s_{1,0,0}s_{\alpha+\beta,\alpha,0} = 
    s_{\alpha+\beta+1,\alpha,0} + 
    s_{\alpha+\beta,\alpha+1,0} +
    s_{\alpha+\beta-1,\alpha-1,0},
    $$
where the second summand does not exists if $\beta = 0$, and the third summand does not exist if $\alpha = 0$.

Rather than deal with the casework of sometimes excluding terms, in using both sums we will still use all three summands. However we still interpret $s_{a,b,c} = 0$ whenever we do not have $a \geq b \geq c \geq 0$.

We have reduced our goal to showing that
    $$
    s_{k+2l+1,k+l+1,0} - s_{k+2l+1,k,0} = 
    \sum n_{\alpha,\beta}H_{\alpha,\beta}
    $$
    where we define
    $$
    H_{\alpha,\beta} = s_{\alpha+\beta+1,\alpha+1,0} - 
    s_{\alpha+\beta,\alpha+1,0} + 
    s_{\alpha+\beta-1,\alpha,0} - 
    s_{\alpha+\beta-1,\alpha-1,0} + 
    s_{\alpha+\beta,\alpha-1,0} - 
    s_{\alpha+\beta+1,\alpha,0}.
    $$
For the remainder of the proof we will assume $k \geq l$ for ease of notation; the case for $k < l$ is exactly analogous.

It will be helpful to extend the notion of $H_{\alpha,\beta}$ to collections of points $(\alpha,\beta)$.  Let $L_i(k,l)$ denote the $i^{th}$ layer of the hexagon corresponding to $k$ and $l$ as described in the statement of Theorem \ref{g2a2thm}.  Note then that $L_i(k,l)$ for $0 \leq i < l$ is the boundary of the hexagon joining the six vertices
    $$(k+l-i,l),(k+l-i,i),(l,i),(i,l),(i,k+l-i),(l,k+l-i),$$
and $L_l(k,l)$ consists of the boundary and interior of the triangle with vertices $(k,l),(l,l),(l,k)$ (or possibly the single point $(l,l)$ if $k = l$).  We then define 
    $$
H_{L_i(k,l)} = s_{k+2l-i+1,k+l-i+1,0} - s_{k+l,k+l-i+1,0} + s_{l+i-1,l,0} - s_{l+i-1,i-1,0} + s_{k+l,i-1,0} - s_{k+2l-i+1,l,0}.
    $$

To better visualize all of  this, let us define $f(\alpha,\beta) = (\alpha+\beta, \alpha)$.  Then $H_{\alpha,\beta}$ consists of six Schur function summands whose corresponding points in the $(\alpha,\beta)$-plane are orthogonally or diagonally adjacent to $f(\alpha,\beta)$, with signs given by the following figure.
    $$
    
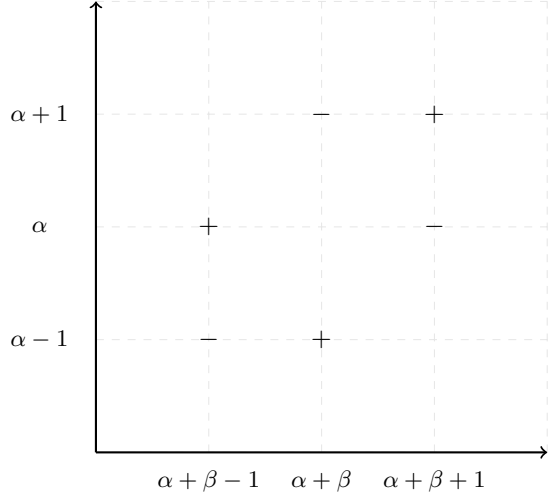
\begin{tikzpicture}[scale=1.5]
\draw[help lines, color=gray!20, dashed]
(-1,-1) grid (3,3);
\draw[->,thick] (-1,-1)--(3,-1);
\draw[->,thick] (-1,-1)--(-1,3);
\draw (0,-1.25) node {\text{\small $\alpha+\beta-1$}};
\draw (1,-1.25) node {\text{\small $\alpha+\beta$}};
\draw (2,-1.25) node {\text{\small $\alpha+\beta+1$}};
\draw (-1.5,0) node {\text{\small $\alpha-1$}};
\draw (-1.5,1) node {\text{\small $\alpha$}};
\draw (-1.5,2) node {\text{\small $\alpha+1$}};
\draw (0,0) node {$-$};
\draw (0,1) node {$+$};
\draw (1,0) node {$+$};
\draw (1,2) node {$-$};
\draw (2,1) node {$-$};
\draw (2,2) node {$+$};
    \end{tikzpicture}
    $$
Likewise, if $S$ is a set of points in the $(\alpha,\beta)$-plane, we define         
    $$f(S)=\{(\alpha+\beta,\alpha) : (\alpha,\beta) \in S\}.$$
In particular, $f(L_i(k,l))$ is the boundary of the hexagon with vertices $$(k+2l-i,k+l-i),(k+l,k+l-i),(l+i,l),(l+i,i),(k+l,i),(k+2l-i,l),$$
and $f(L_l(k,l))$ is the boundary and interior of the triangle with vertices $(k+l,k),(2l,l),(k+l,l)$ (or just the single point $(2l,l)$ if $k = l$). Note that, for any $0 < i \leq l$, the summands of $H_{L_i(k,l)}$ correspond to the vertices of $f(L_{i-1}(k,l))$.

\begin{lem}
For $k\geq l$,
    $$
    \sum_{(\alpha,\beta) \in L_l(k,l)} H_{\alpha,\beta} = H_{L_l(k,l)}.
    $$
\end{lem}

\begin{proof}
Let $k = l+j$. We proceed by induction on $j$.  The cases $j = 0$ and $j = 1$ are easy to verify for any $l$. For the inductive step, suppose that for a fixed $j = J$ we have established the lemma.  We now wish to show that
    $$
    \sum_{(\alpha,\beta) \in L_l(l+J+1,l)} H_{\alpha,\beta} = H_{L_l(l+J+1,l)}.
    $$
    By the inductive hypothesis we have that 
    \begin{align*}
        \sum_{(\alpha,\beta) \in L_l(l+J,l)} H_{\alpha,\beta} 
        &= H_{L_l(l+J,l)}\\
        &= 
    s_{2l+J+1,l+J+1,0} - s_{2l+J,l+J+1,0} + s_{2l-1,l,0} - s_{2l-1,l-1,0} + s_{2l+J,l-1,0} - s_{2l+J+1,l,0}.
    \end{align*}
    To expand this sum to include all integer points of $L_l(l+J+1,l)$, we must add 
    $$
    \sum_{i=0}^{J+1} H_{l+i,l+J+1-i}.
    $$
    Doing so adds the following terms
    $$
    \sum_{i=0}^{J+1} (s_{2l+J,l+i,0}-s_{2l+J,l-1+i,0}) +
    (s_{2l+J+1,l-1+i,0}-s_{2l+J+1,l+1+i,0}) + 
    (s_{2l+J+2,l+1+i,0} - s_{2l+J+2,l+i,0}) 
    $$
to the previous sum, which equals the following eights terms
    \begin{align*}
    &s_{2l+J,l+J+1,0} - s_{2l+J,l-1,0} 
    + s_{2l+J+1,l-1,0} + s_{2l+J+1,l,0} \\
    &\qquad - s_{2l+J+1,l+J+1,0} - s_{2l+J+1,l+J+2,0}
    + s_{2l+J+2,l+J+2,0} - s_{2l+J+2,l,0}
    \end{align*}
after telescoping.

    Adding these terms to the sum in the inductive hypothesis and cancelling gives
    \begin{align*}
    \sum_{(\alpha,\beta) \in L_l(l+J+1,l)} H_{\alpha,\beta} &= 
        \begin{aligned}[t]
            &s_{2l+J+2,l+J+2,0} - s_{2l+J+1,l+J+2,0} +
    s_{2l-1,l,0} \\
    &\qquad - s_{2l-1,l-1,0} + 
    s_{2l+J+1,l-1,0} - s_{2l+J+2,l,0}
        \end{aligned}\\
    &= 
    H_{L_l(l+J+1,l)},
    \end{align*}
    as desired.
\end{proof}

We now need to prove an analogous lemma for the hexagonal layers.
 
\begin{lem}

For $0 \leq i < l$ we have
$$
\sum_{(\alpha,\beta) \in L_{i}(k,l)} H_{\alpha,\beta} = H_{L_{i+2}(k,l)} - 2H_{L_{i+1}(k,l)} + H_{L_{i}(k,l)}.
$$
(In the case $i = l-1$ we define $H_{L_{l+1}(k,l)} = 0$.)
 
\end{lem}

\begin{proof}
The sum on the left-hand side can be viewed as taking the hexagon in our figure and sliding it along the hexagon defined by $f(L_i(k,l))$.  For any given $i$, note that any summand produced by the sum on the left hand side must be on or adjacent to $f(L_i(k,l))$, so the only Schur functions that can occur correspond to points on one of $f(L_{i+1}(k,l))$, $f(L_i(k,l))$, or $f(L_{i-1}(k,l))$ (where when $i = 0$, we define $L_{-1}(k,l)$ to be the hexagon surrounding $L_0(k,l)$ in the appropriate way).  All of these are hexagons except when $i = l-1$, in which case $f(L_l(k,l))$ is a triangle.  

The proof strategy is to split the points $(a,b)$ in these three layers into cases, and determine how often and with what sign each point occurs in some $H_{\alpha,\beta}$ in our sum. 

\textit{Case 1}:  The point $(a,b)$ lies on $f(L_{i+1}(k,l))$. 

Subcase 1.1: $i = l-1$.  In this case, if $k = l$, then $f(L_l(l,l))$ is the single point $(2l,l)$.  Then there are six $H_{\alpha,\beta}$ terms of our sum that produce $s_{2l,l,0}$, one for every point on the hexagon $L_{l-1}(l,l)$, and the summand $s_{2l,l,0}$ will appear in three of these terms with a positive sign and in three with a negative sign, and thus will get a total coefficient of 0. 

If $k > l$ then clearly we need only look at points on the boundary of the triangle $f(L_l(k,l))$.  Any point $(a,b)$ on the boundary that is not a vertex will be adjacent to three points of $f(L_{l-1}(k,l))$, and its corresponding summand $s_{a,b,0}$ will appear in two of the corresponding $H_{\alpha,\beta}$ of the sum with opposite signs, and so will vanish.  This is easy to check for each side of the triangle $f(L_l(k,l))$ separately; for example, if we take a point $(l+j,l) = f(l,j)$ for $l < j < k$, this will be adjacent to the three points 
    $$(l+j-1,l-1),(l+j,l-1),(l+j+1,l-1)$$
of $f(L_{l-1}(k,l))$, which correspond to $f(l-1,j),f(l-1,j+1),f(l-1,j+2)$. Of the corresponding terms in our sum, $H_{l-1,j}$ will contain $+s_{l+j,l,0}$, $H_{l-1,j+1}$ will contain $-s_{l+j,l,0}$, and $H_{l-1,j+2}$ does not contain $s_{l+j,l,0}$ at all.

If $(a,b)$ is a vertex of $f(L_l(k,l))$, then it is one of $f(k,l),f(l,l),$ or $f(l,k)$.  Again, we can check directly that if we take all adjacent points $(c,d) \in f(L_{l-1}(k,l))$ and sum the appearances of $s_{a,b,0}$ in $H_{f^{-1}(c,d)}$ we will get $0$.  Thus in the case $i = l-1$, our sum produces no Schur functions corresponding to points on $f(L_l(k,l))$, justifying our defining $H_{L_{l+1}(k,l)} = 0$.

Subcase 1.2: $i < l-1$.  Then $f(L_{i+1}(k,l))$ is a hexagon, and we consider any $(a,b)$ lying on it.  If $(a,b)$ is not a vertex of $f(L_{i+1}(k,l))$, it is adjacent to three points of $f(L_i(k,l))$, and it is easy to check that its summand will occur in two of the corresponding $H_{\alpha,\beta}$ with opposite signs and thus will vanish; this can be verified for each side of the hexagon separately like in Subcase 1.1.  

If $(a,b)$ is a vertex of $f(L_{i+1}(k,l))$, we end up with a different result than in Subcase 1.1.  These vertices are          
    $$(k+2l-i-1,k+l-i-1),(k+l,k+l-i-1),(l+i+1,l),(l+i+1,i+1),(k+l,i+1),(k+2l-i-1,l).$$
It can be checked for each of these points $(a,b)$ separately that if we look at the set of adjacent $(c,d)$ in $f(L_i(k,l))$, and sum over $H_{f^{-1}(c,d)}$ the coefficient of $s_{a,b,0}$, and then sum those results together, we get
    $$
s_{k+2l-i-1,k+l-i-1,0} - s_{k+l,k+l-i-1,0} + s_{l+i+1,l,0} - s_{l+i+1,i+1,0} + s_{k+l,i+1,0} - s_{k+2l-i-1,l,0} = H_{L_{i+2}(k,l)}
    $$
Thus, considering only terms corresponding to points of $f(L_{i+1}(k,l))$, we get that  $\sum_{(\alpha,\beta) \in L_i(k,l)} H_{\alpha,\beta}$ yields $H_{L_{i+2}(k,l)}$.

\textit{Case 2}: The point $(a,b)$ lies on $f(L_i(k,l))$.  If $(a,b)$ is not a vertex of $f(L_i(k,l))$, then it will be adjacent to two other points of $f(L_i(k,l))$, and will appear in the two corresponding $H_{\alpha,\beta}$ of the sum with opposite signs and vanish.  This can be checked on each side of the hexagon separately as in the previous case.

If $(a,b)$ is a vertex of $f(L_i(k,l))$, it is one of six possibilities.  We can again check the contribution of the $H_{\alpha,\beta}$ to each point separately to determine that, when restricted to $f(L_i(k,l))$, we get $-2H_{L_{i+1}(k,l))}$ out of $\sum_{(\alpha,\beta) \in L_i(k,l)} H_{\alpha,\beta}$.

\textit{Case 3}: The point $(a,b)$ lies on $f(L_{i-1}(k,l))$.  If $(a,b)$ is not a vertex of $f(L_{i-1}(k,l))$, then as with the other cases, it will appear in two $H_{\alpha,\beta}$ of the sum with opposite signs and vanish.  This can be checked on each side of the hexagon separately.

If $(a,b)$ is a vertex of $f(L_{i-1}(k,l))$, then as with the other cases, we can look at the contribution of the $H_{\alpha,\beta}$ of the adjacent points of $f(L_i(k,l))$ for each vertex separately, sum them, and get a contribution of $H_{L_i(k,l)}$.

Taking all three cases together, we have examined every point $(a,b)$ such that $s_{a,b,0}$ occurs as a summand (with either sign) of some $H_{\alpha,\beta}$ in our sum, and determined the coefficient of $s_{a,b,0}$ arising from summing over all applicable $H_{\alpha,\beta}$.  Thus, adding the results together from the three cases, we may conclude that
$$
\sum_{(\alpha,\beta) \in L_i(k,l)} H_{\alpha,\beta} = H_{L_{i+2}(k,l)} - 2H_{L_{i+1}(k,l)} + H_{L_i(k,l)},
$$
as desired.
\end{proof}

Now we can put this all together to finish proving Theorem \ref{g2a2thm}.  Let $H(k,l) = \bigcup_{i=0}^l L_i(k,l)$ be the hexagon of Theorem \ref{g2a2thm} associated with $(k,l)$.  Furthermore, using our new terminology $n_{\alpha,\beta} = i+1$ when $(\alpha,\beta) \in L_i(k,l)$. We wish to evaluate 
    $$
\sum_{(\alpha,\beta) \in H(k,l)} n_{\alpha,\beta}H_{\alpha,\beta} = \sum_{i=0}^l (i+1)\sum_{(\alpha,\beta) \in L_i(k,l)} H_{\alpha,\beta}.
    $$
Simplifying this using our lemmas gives us
    $$
(l+1)H_{L_l(k,l)} + \sum_{i=0}^{l-1} (i+1)\left( H_{L_{i+2}(k,l)} - 2H_{L_{i+1}(k,l)} + H_{L_{i}(k,l)}\right).
    $$
We can check that for each $2 \leq j \leq l$ the coefficient of $H_{L_j(k,l)}$ is $(j+1)-2j+(j-1)=0$, and for $j = 1$ the coefficient of $H_{L_1(k,l)}$ is also $0$.  The only summand that does not get cancelled is a single 
    $$
    H_{L_0(k,l)} = 
    s_{k+2l+1,k+l+1,0} - s_{k+l,k+l+1,0} +
    s_{l-1,l,0} - s_{l-1,-1,0} +
    s_{k+l,-1,0} - s_{k+2l+1,l,0}
    $$
The middle four terms in this expression are $0$, so all that remains is the desired $s_{k+2l+1,k+l+1,0} - s_{k+2l+1,l,0}$, and we are done.

\section{Conjectures for other cases} \label{here}

In the final section of the paper we provide some conjectural descriptions for the Types (III) and (V) branching rules. These conjectures are combinatorial in nature, in agreement with Conjecture \ref{conjj}, and are derived by comparing dimensions between irreducible representations of $\Delta$ and $\Delta_1$. 


\subsubsection*{Type (III) with $r=3$}

Let $A = a+1, B = b+1, C= c+1$, where $a,b,c$ are positive integers. Then the irreducible representation of highest weight $(a,b,c)$ in $C_3$ has dimension
$$\dim\Pi_{a,b,c} = \frac{ABC}{720}(A+B)(B+C)(A+B+C)(B+2C)(A+B+2C)(A+2B+2C).$$
For $(A_1)^3$, one has
    $$\dim \pi_{a,b,c}=ABC.$$ 
Let $\Pi$ denote a representation of $C_3$, and $\pi$ one of $(A_1)^{\oplus 3}$.
First and most simply we have the formula
    $$\nores\Pi_{a,0,0} = \bigoplus_{r+s+t=a} \pi_{r,s,t}$$
If we let $T_k$ be the set of triples $(r,s,t)$ of integers $0 \leq r,s,t \leq k$ with $r+s+t = 2k$, then
    $$\nores \Pi_{0,b,0} = \bigoplus_{k=0}^b \bigoplus_{(\tau_1,\tau_2,\tau_3) \in T_k} (b-k+1)\pi_{\tau_1,\tau_2,\tau_3}.$$
We can in fact write down a branching rule for $\nores \Pi_{a,b,0}$ that is very reminiscent of the one from $G_2$ to $A_2$. Assume $a, b > 0$, so that all three coordinates of $(a+b,b,0)$ are distinct. Then the formula for $\nores \Pi_{a,b,0}$ will be a direct sum of irreducible representations of $(A_1)^{\oplus 3}$ with highest weights counted as follows: 
    \begin{itemize}
        \item Form a hexagon with the six vertices that are permutations of the coordinates of $(a+b,b,0)$, connected in such a way that the resulting hexagon is convex. (This hexagon will be parallel to the plane $x+y+z=0$.) We count with multiplicity $1$ every representation on the perimeter of this hexagon, with multiplicity $2$ every point that is one coordinate orthogonally from the perimeter, and so on until we reach an inner layer whose points form the perimeter of a triangle; then all points remaining on this perimeter and inside the triangle all receive the same multiplicity.
        \item Now, form a smaller hexagon whose vertices are $(a+b-1,b-1,0)$ and all its other coordinate permutations. (It is possible that some of these values will be the same, so that there are only three permutations; in that case, we just construct the triangle with those vertices.)  We now double our multiplicity count by giving those entries on the perimeter multiplicity $2$, those on the first layer inside multiplicity $4$, and so on until reaching an interior triangle all of whose points get the same multiplicity.  
        \item Iterating, we keep forming smaller hexagons (or possibly triangles) with vertices $(a+b-k,b-k,0)$ and its coordinate permutations as long as $b-k \geq 0$.  The perimeter of the level $k$ hexagon gets multiplicity $k+1$, the first layer inside the perimeter gets multiplicity $2(k+1)$, and so on, until a layer is reached that is a triangle, after which all points receives the same multiplicity. 
    \end{itemize} 
Furthermore, we claim that 
$$\nores \Pi_{0,0,c} = \bigoplus_{k=0}^c \bigoplus_{(\tau_1,\tau_2,\tau_3) \in T_k} \pi_{c-\tau_1,c-\tau_2,c-\tau_3}.$$
At present we do not see a general formula for $\nores \Pi_{a,b,c}$, nor can we write down an explicit formula for the branching rules of Types (III) with general $r\geq 3$.

\subsubsection*{Type (V)}

One can try to understand $\Ree$ by understanding $\Rees$, for there are the embeddings $D_4$ into $B_4$, and $B_4$ into $F_4$. The highest weight representations of the root systems $F_4,B_4,D_4$ are parameterized by four integers in terms of fundamental weights, and we denote their highest weight representations by $\Pi,\rho,\pi$ respectively. In the book \cite{mpr} the restrictions from $F_4$ to $B_4$ are listed for small highest weights. This suggests the following two branching rules for $\Rees$ as we vary the weights corresponding to the first and last vertices of the $F_4$ Dynkin diagram:
    \begin{align*}
    \Ree \Pi_{k,0,0,0}&= \bigoplus_{s+t=k} \rho_{0,s,0,t}, \\
    \Ree \Pi_{0,0,0,k}&= \bigoplus_{0\leq s+t\leq k} \rho_{s,0,0,t}.
    \end{align*}
We now refer to the Gelfand-Tsetlin pattern of Type (II) stated in Section \ref{srssec}. Following the notations of that section, the correspondence between $\Pi^r,\pi^r$ and $\rho,\pi$, for the root systems $B_4,D_4$ respectively, can be stated as follow:
    \begin{align*}
    \Pi^r_{f_1,f_2,f_3,f_4} &= \rho_{f_1-f_2,f_2-f_3,f_3-f_4,2f_4}, \\
    \pi^r_{g_1,g_2,g_3,g_4} &= \pi_{g_1-g_2,g_2-g_3,g_3-g_4,g_3+g_4}.
    \end{align*}
Using this, it is not hard to see that
    \begin{align*}
    \Reese \rho_{0,s,0,t} &= \bigoplus_{\substack{s'+s''= s \\ 0\leq t'\leq t}} \pi_{s',s'',t',t-t'}, \\
    \Reese \rho_{s,0,0,t} &= \bigoplus_{\substack{0\leq s'\leq s \\ t'+t''=t}} \pi_{s',0,t',t''}.
    \end{align*}
Therefore
    \begin{align*}
    \Ree \Pi_{k,0,0,0}&= \bigoplus_{0\leq s'+s''+t'\leq k} \pi_{s',s'',t',k-s'-s''-t'}, \\
    \Ree \Pi_{0,0,0,k}&= \bigoplus_{0\leq s'+t'+t''\leq k} (k+1-s'-t'-t'') \pi_{s',0,t',t''}.
    \end{align*}

\begin{exmp}
We list here four examples of $\Reese$. The last example is not included as a special case of the two formulas presented above.
    \begin{align*}
    \Ree \Pi_{0,0,0,1} &= 2 \pi_{0,0,0,0} \oplus (\pi_{1,0,0,0} \oplus \pi_{0,0,1,0} \oplus\pi_{0,0,0,1}) \\
    \Ree \Pi_{1,0,0,0} &= \pi_{0,1,0,0}\oplus (\pi_{1,0,0,0} \oplus \pi_{0,0,1,0}\oplus \pi_{0,0,0,1}) \\
    \Ree \Pi_{2,0,0,0} &= 
        \begin{aligned}[t]
        \pi_{0,2,0,0}&\oplus (\pi_{2,0,0,0} \oplus \pi_{0,0,2,0}\oplus \pi_{0,0,0,2})\oplus (\pi_{1,1,0,0}\oplus \pi_{0,1,1,0} \oplus \pi_{0,1,0,1}) \\
        &\oplus (\pi_{1,0,1,0}\oplus \pi_{1,0,0,1}\oplus \pi_{0,0,1,1})
        \end{aligned}\\
    \Ree \Pi_{0,0,1,0} &=2 \pi_{0,1,0,0}\oplus  2 (\pi_{1,0,0,0} \oplus \pi_{0,0,1,0}\oplus \pi_{0,0,0,1})\oplus (\pi_{1,0,1,0}\oplus \pi_{1,0,0,1}\oplus \pi_{0,0,1,1}) \oplus \pi_{0,0,0,0} 
    \end{align*}
Note that the branching rule in the above examples possesses the following symmetry: if $\pi_{a,b,c,d}$ appears with multiplicity $m$, then so does $\pi_{\sigma(a),b,\sigma(c),\sigma(d)}$ for all $\sigma\in S_3$. 
\end{exmp}

\section{Acknowledgments}

The authors would like to thank Tao Song for many helpful discussions, and George Wang, for pointing out a way to simplify the proof of Theorem 8.

\end{document}